\documentclass[12pt]{amsart}

\usepackage{amsmath, amscd}
\usepackage{amssymb}
\usepackage{enumerate}
\usepackage[pagebackref,colorlinks=true,linkcolor=blue,citecolor=blue]{hyperref}

\newcommand{\ga}{\gamma}

\newcommand{\A}{\mathbb{A}}
\newcommand{\QQ}{\mathbb{Q}}

\newcommand{\CC}{\mathbb{C}}

\numberwithin{equation}{section}
 \theoremstyle{plain}
\newtheorem{theorem}{Theorem}[section]
\newtheorem{lemma}[theorem]{Lemma}

\newtheorem{proposition}[theorem]{Proposition}
\theoremstyle{definition}
\newtheorem{definition}[theorem]{Definition}
\newtheorem{convention}[theorem]{Convention}
\theoremstyle{remark}
\newtheorem*{rem}{Remark}

\newcounter{remarkscounter}
\newenvironment{remarks}
{\medskip\noindent{\it
Remarks.}\begin{list}{{\rm(\arabic{remarkscounter})}
}{\usecounter{remarkscounter}

\setlength{\labelsep}{\fill} \setlength{\leftmargin}{0pt}
\setlength{\itemindent}{\fill}
\setlength{\labelwidth}{\fill}\setlength{\topsep}{0pt}
\setlength{\listparindent}{0pt}}} {\end{list}}

\begin{document}

\title{A simple twisted relative trace formula}
\author{Heekyoung Hahn}

\address{Department of Mathematics and Statistics, University at Albany, Albany, NY 12222}
\email{hhahn@albany.edu}


\begin{abstract}
In this article we derive a simple twisted relative trace formula.
\end{abstract}

\maketitle

\section{Introduction}

In this paper, we obtain a simple twisted relative trace formula analogous to the simple trace formula of Deligne and Kazhdan  \cite{BDKV} (see also \cite{Ro}).  We hope this simple twisted relative trace formula can be used to study distinguished admissible representations of a reductive group. A similar formula has already proved useful in the study of distinguished admissible representations of general linear groups \cite{HM}. Our formula is also used  crucially in \cite{GW}. We will state the main result of \cite{GW} after stating our main theorem.

Let $M/F$ be a quadratic or trivial extension of a number field $F$ and $\tau$ be the generator of $\text{Gal}(M/F)$. Assume that $G_0$ is a connected reductive $F$-group with semisimple automorphism $\sigma$ of order $2$. We denote by  $G_0^{\sigma}$ the group whose points in an $F$-algebra $R$ are given by
$$G_0^{\sigma}(R)=\{ g\in G_0(R) : g^{\sigma}=g\}.$$Then $G_0^{\sigma}$  is reductive \cite[\S 9]{S2}. Write
\begin{align*}
G&:=\text{Res}_{M/F}G_0,\\
G^{\sigma}&:=\text{Res}_{M/F}G_0^{\sigma},
\end{align*}set $\theta=\sigma\circ\tau$, and write $G^{\theta}\leq G$ for the subgroup whose points in an $F$-algebra $R$ are given by
$$G^{\theta}(R)=\{ g\in G(R) : g^{\theta}=g\}\,.$$Note that $\sigma$ and $\tau$ commute.

In order to state our main result, we define the notion of an $F$-supercuspidal function.

\begin{definition}
Let $v$ be a finite place of $F$ and let $f_v\in C_c^{\infty}(G(F_v))$. We say that $f_v$ is $F$-{\it supercuspidal} if any translate of $f_v$ (left or right) has zero integral along the unipotent radical of any proper parabolic of $G_{F_v}$ that is defined over $F$; i.e.~is the base change of a parabolic subgroup of $G$.
\end{definition}

Let $$\mu:=\mu_1\times\mu_2: G^{\sigma}(\mathbb{A}_F)\times G^{\theta}(\mathbb{A}_F)\longrightarrow \mathbb{C^{\times}}$$ be a unitary character trivial on $G^{\sigma}(F)\times G^{\theta}(F)$.
We now state our main result using notations that are defined precisely in \S 2 below.

\begin{theorem}\label{main}
Suppose that $G^{\sigma}$ is connected. Let $f=\otimes_v f_v\in C_c^{\infty}(G(\mathbb{A}_F))$ be such that there exist places $v_1, v_2, v_3$ of $F$ (not necessarily distinct) such that
\begin{enumerate}
\item{$f_{v_1}$ is supported on relatively $\tau$-elliptic elements of $G(F_{v_1})$,}
\item{$f_{v_2}$ is supported on strongly relatively $\tau$-regular elements of $G(F_{v_2})$,}
\item{$f_{v_3}$ is $F$-supercuspidal.}
\end{enumerate}
Then we have that
\begin{align}\label{main-eqn}
\sum_{\{\gamma\}}c^{\tau}(\gamma) \mathrm{TRO}^{\mu}_{\gamma}(f)=\sum_{\pi}\mathrm{TRT}_{G^{\sigma}, G^{\theta}}^{\mu}(\pi(f^1)),
\end{align}where the sum on the left is over relevant strongly relatively $\tau$-regular elliptic  classes in $G(F)$ and the sum on the right is over equivalence classes of automorphic representations $\pi$ of ${}^1G(\mathbb{A}_F)$ admitting a realization in $L_0^2(G(F)\backslash {}^1G(\mathbb{A}_F)$. Here $f^1$ is a normalization of $f$ defined in \eqref{f1} and ${}^1G(\mathbb{A}_F)$ is the Harish-Chandra subgroup of $G(\mathbb{A}_F)$(see \S 2).
\end{theorem}

\begin{remarks}
\item{In the theorem, $c^{\tau}(\gamma)$ is a volume term, $\mathrm{TRO}^{\mu}_{\gamma}(f)$ is a twisted relative orbital integral as in \eqref{TRO}, and $\mathrm{TRT}_{G^{\sigma}, G^{\theta}}^{\mu}(\pi(f^1))$ is a twisted relative trace as in \eqref{TRT}.} 
\item{If $f^1$ is $K_{\infty}$-finite, where $K_{\infty}$ is a maximal compact subgroup of $G(\A_{F,\infty})$, then 
$$
\mathrm{TRT}_{G^{\sigma},G^{\theta}}^{\mu}(\pi(f^1))=\sum_{\phi \in \mathcal{B}_{\pi}} \mathcal{P}^{\sigma}_{\mu_1}(\pi(f^1)\phi) \overline{\mathcal{P}^{\theta}_{\mu_2}(\phi)},
$$
where $\mathcal{P}^{\sigma}_{\mu_1}(\phi)$ and $\mathcal{P}^{\theta}_{\mu_2}(\phi)$ are period integrals of $\phi$ over $G^{\sigma}$ and $G^{\theta}$, respectively, (see \eqref{Period_int1} and \eqref{Period_int2} for definitions). Here $\mathcal{B}_{\pi}$ is an orthonormal basis of the $\pi$-isotypic subspace of $L_0^2(G(F)\backslash {}^1G(\mathbb{A}_F))$ with respect to the pairing \eqref{pair} that consists of smooth $K_{\infty}$-finite vectors.}
\item{The condition that $G^{\sigma}$ is connected in the theorem is used to prove Proposition \ref{one} and to apply \cite[\S 2, Proposition 1]{AGR} to prove Proposition \ref{spectral}.}
\end{remarks}

\medskip
Theorem \ref{main} will be restated and proven in \S \ref{Se}. The proof of the theorem roughly follows \cite{BDKV}. However, our proof uses Galois cohomology in place of reduction theory to prove that the geometric side of the twisted relative trace formula is a finite sum.

In \cite{GW}, our simple twisted relative trace formula is used to prove the following result: Let $U'$ be a quasi-split unitary group of odd rank $n$ with respect to a CM extension $M/F$, let $E/F$ be a totally real quadratic extension with $\langle \sigma \rangle=\text{Gal}(E/F)$, and let $U:=\mathrm{Res}_{E/F}U'$. Under suitable local hypotheses, a cuspidal cohomological automorphic representation $\pi$ of $U$ such that $\pi^{\vee} \cong \pi^{\sigma}$ is nearly equivalent to cohomological automorphic representation $\pi'$ of $U$ that is $U'$-distinguished in the sense that there is a form in the space of $\pi'$ admitting a nonzero period over $U'$.

\section{Preliminaries}

\subsection{Algebraic groups}
Let $G_0$ be an algebraic group over a field $F$.  We write $Z_{G_0}$ for the center of $G_0$ and $G_0^{\circ}$ for the identity component of $G_0$ in the Zariski topology. If $\gamma\in G_0(F)$ and $G'_0\leq G_0$ is a subgroup we write $C_{G_0', \gamma}$ for the centralizer in $G'_0$ of $\gamma$. 

Let $M/F$ be a quadratic extension of $F$-algebras. Thus $M$ is either a field or $M=F\oplus F$.
Let $G:=\text{Res}_{M/F}G_0$ and let $\tau$ be the automorphism of $G$ induced by the generator of $\text{Gal}(M/F)$. Assume $G_0$ (and hence $G$) is connected. As usual, an element $\gamma\in G(F)$ is said to be $\tau$-semisimple if $\ga\ga^{\tau}$ is semisimple. A $\tau$-semisimple element $\ga$ is $\tau$-elliptic (resp. $\tau$-regular semisimple, $\tau$-regular elliptic) if $C^{\tau}_{G, \gamma}/Z_{G_0}$ is anisotropic (resp. a torus, an anisotropic torus).

If $\sigma$ is an automorphism of $G$ and $\gamma\in G(R)$ for some commutative $F$-algebra $R$ we write
$$\gamma^{-\sigma}:=(\gamma^{-1})^{\sigma}\,.$$

\subsection{Ad\`{e}les}
The ad\`{e}les of a number field $F$ will be denoted by $\mathbb{A}_F$. For a set of places $V$ of $F$ we write $\mathbb{A}_{F, V}:=\otimes'_{v\in V}F_v$ and $\mathbb{A}_F^V:=\otimes'_{v\notin V} F_v$. Here $\otimes'$ denotes the typical restricted direct product. The set of infinite places of $F$ will be denoted by $\infty$. For an affine $F$-variety $G$ and a subset $W\leq G(\mathbb{A}_F)$ the notation $W_V$ (resp. $W^V$) will denote the projection of $W$ to $G(\mathbb{A}_{F, V})$ (resp. $G(\mathbb{A}_F^V)$). If $W$ is replaced by an element of $G(\mathbb{A}_F)$, or if $G$ is an algebraic group and $W$ is replaced by a character of $G(\mathbb{A}_F)$ or a Haar measure on $G$, the analogous notation will be in force; for example, if $\gamma\in G(\mathbb{A}_F)$ then $\gamma_v$ is the projection to $G(F_v)$ of $\gamma$.

\subsection{Harish-Chandra subgroups}
Let $G$ be a connected reductive group over a number field $F$. We write $A_G\leq Z_G(F\otimes_{\mathbb{Q}}\mathbb{R})$ for the connected component of the real points of the largest $F$-split torus in the center of $G$. Write $X^*$ for the group of $F$-rational characters of $G$ and write $\mathfrak{a}_G:=\text{Hom} (X^*, \mathbb{R})$. The Harish-Chandra morphism $$HC_G: G(\mathbb{A}_F)\longrightarrow \mathfrak{a}_G$$is defined by
$$\langle HC_G(x), \chi\rangle=|\log (x^{\chi})|$$for $x\in G(\mathbb{A}_F)$ and $\chi\in X^*$. We write
\begin{equation}\label{ker}
{}^1G(\mathbb{A}_F):=\mathrm{ker}(HC_G)
\end{equation}and refer to it as the Harish-Chandra subgroup. Note that $G(\mathbb{A}_F)$ is the direct product of $A_G$ and ${}^1G(\mathbb{A}_F)$, and that $G(F)\leq {}^1G(\mathbb{A}_F)$.

\subsection{Definitions and notations}

Let $G_0$ be a connected reductive $F$-group with automorphism $\sigma$ of order $2$ and let
$G_0^{\sigma}$, $G$, etc.~be defined as in the introduction.
We have left actions of $G_0^{\sigma} \times G_0^{\sigma}$ on $G_0$ and $G^{\sigma} \times G^{\theta}$ on $G$ given by
\begin{align}
(G_0^{\sigma}\times G_0^{\sigma})(R) \times G_0(R) &\longrightarrow G_0(R)\label{action_G0}\\
((g_1,g_2), g) &\longmapsto g_1gg_2^{-1},\nonumber\\
(G^{\sigma} \times G^{\theta})(R) \times G(R) &\longrightarrow G(R)\label{action_G}\\
((g_1,g_2),g) &\longmapsto  g_1gg_2^{-1}\nonumber
\end{align}
for $F$-algebras $R$. Following \cite[Lemma 2.4]{Ri} we introduce the subschemes $Q\subset G_0$ and $S\subset G$ defined as the scheme theoretic images of the morphisms given on $R$-points by
\begin{align}
B_{\sigma}\,:\,G_0(R) &\longrightarrow G_0(R)\label{Q}\\
 g&\longmapsto gg^{-\sigma}\,, \nonumber\\
B_{\theta}\,:\,G(R) &\longrightarrow G(R)\label{S}\\
 g&\longmapsto gg^{-\theta}\,, \nonumber
\end{align}
respectively. Here $R$ is a (commutative) $F$-algebra. The $F$-schemes $Q$ and $S$ are closed affine $F$-subschemes of $G_0$ and $G$, respectively (see \cite[\S 2.1]{HW}). Note that
$$
B_{\sigma}((g_1, g_2)\cdot g)=g_1B_{\sigma}(g)g_1^{-1}\quad\text{and}\quad
B_{\theta}((g_1, g_2)\cdot g)=g_1B_{\theta}(g)g_1^{-\tau},
$$where $\cdot$ denotes the left actions defined by \eqref{action_G0} and \eqref{action_G} respectively.

\medskip
\noindent The following definition is an adaptation of one appearing in \cite{F} and \cite{JL} (see \cite{GW}):
\begin{definition}
Let $k$ be a field extension of $F$. An element $\gamma\in G(k)$ is relatively $\tau$-semisimple (resp. relatively $\tau$-elliptic, relatively $\tau$-regular semisimple, relatively $\tau$-regular elliptic) if $\gamma\gamma^{-\theta}$ is $\tau$-semisimple (resp. $\tau$-elliptic,  $\tau$-regular semisimple, $\tau$-regular elliptic) as an element of $G(k)$.
\end{definition}

\begin{definition}
For any $F$-algebra $k$ and any $\gamma\in G(k)$, write $G^{\sigma\theta}_{\gamma}$ for the $k$-group whose points in a commutative $k$-algebra $R$ are given by
\begin{equation}\label{G_sigma_theta}
G^{\sigma\theta}_{\gamma}(R):=\{ (x,y)\in G^{\sigma}(R)\times G^{\theta}(R) : x^{-1}\gamma y=\gamma\}.
\end{equation}
Write $C^{\tau}_{G^{\sigma}, \gamma}$ for the $\tau$-centralizer of $\gamma$ in $G^{\sigma}$.  Explicitly, $C^{\tau}_{G^{\sigma},\gamma}$ is the $k$-group whose points in a commutative $k$-algebra $R$ are given by
\begin{equation}\label{centralizer}
C^{\tau}_{G^{\sigma}, \gamma}(R):=\{ x\in G^{\sigma}(R) : x\gamma x^{-\tau}=\gamma \}.
\end{equation}
\end{definition}

\begin{definition} Let $k$ be an $F$-algebra.  A relatively $\tau$-semisimple element $\gamma \in G(k)$  is strongly relatively $\tau$-regular if $C^{\tau}_{G^{\sigma}, \ga\ga^{-\theta}}$ is a (connected)  torus.
\end{definition}
\noindent Thus a strongly relatively $\tau$-regular element is relatively $\tau$-regular semisimple.

\medskip

For an $F$-algebra $R$, write
\begin{equation}\label{class}
\Gamma_r(R):=G^{\sigma}(R)\backslash G(R)/G^{\theta}(R)\,.
\end{equation}We refer to the elements of $\Gamma_r(R)$ as {\bf relative $\tau$-classes}. If two elements $\gamma, \gamma'\in G(R)$ map to the same class in $\Gamma_r(R)$, we say that they are in the same relative $\tau$-class.

Let $k$ be $F$ or $F_v$ for some place $v$ of $F$. Then we say that two relatively $\tau$-semisimple $\gamma, \gamma'\in G(k)$ are in the same {\bf geometric relative $\tau$-class} if there exists $(\overline{x}, \overline{y})\in G^{\sigma}(\overline{k})\times G^{\theta}(\overline{k})$ such that $\overline{x}^{-1}\gamma \overline{y}=\gamma'$. We write $\Gamma_r^{\mathrm{ge}}(k)$ for the set of geometric relative $\tau$-classes.  Notice that if $\gamma$ and $\gamma'$ are in the same geometric relative $\tau$-class then $\gamma$ is strongly relatively $\tau$-regular  if and only if $\gamma'$ is as well.

\begin{lemma}\label{closed}
Let $k$ be $F$ or $F_v$ for some place $v$ of $F$. If $\gamma\in G(k)$ is relatively $\tau$-semisimple then the $G^{\sigma}$-orbit of $\gamma\gamma^{-\theta}$ in $S_k$ is closed.
\end{lemma}

\begin{proof} If $\tau$ is trivial, this is Theorem 7.5 of \cite{Ri}, so we assume for the remainder of the proof that $\tau$ is nontrivial. Let $\overline{k}$ be an algebraic closure of $k$.
By faithfully flat descent along $k\to \overline{k}$, it is enough to prove that the orbit of $\gamma\gamma^{-\theta}$ under $G^{\sigma}_{\overline{k}}$ is closed. Choose an isomorphism $G_{\overline{k}}\cong G_{0\overline{k}}\times G_{0\overline{k}}$ intertwining $\tau$ with $(x, y)\longmapsto (y, x)$ and equivariant with respect to $\sigma$. Note that there is a morphism $S_{\overline{k}}\to Q_{\overline{k}}$ given at the level of $R$-points by
\begin{align*}
N : S_{\overline{k}}(R)&\longrightarrow Q_{\overline{k}}(R)\\
(x_1, x_2)&\longmapsto x_1x_2\,,
\end{align*}for $\overline{k}$-algebra $R$. The element $N(\gamma\gamma^{-\theta})\in Q(\overline{k})$ is semisimple by assumption, so by Theorem 7.5 of \cite{Ri} we conclude that its $G^{\sigma}_{0\overline{k}}$ conjugacy class is closed. The inverse image of the $G^{\sigma}_{0\overline{k}}$ conjugacy class of $N(\gamma\gamma^{-\theta})$ is the $G^{\sigma}_{\overline{k}}$ orbit of $\gamma\gamma^{-\theta}$ (under $\tau$-conjugation). The $G^{\sigma}_{\overline{k}}$ orbit of $\gamma\gamma^{-\theta}$ under $\tau$-conjugation is therefore closed since it is the inverse image of a closed set under a continuous map.
\end{proof}

\begin{theorem}\label{Stein}
If $\gamma\in G(k)$ is relatively $\tau$-semisimple, then $G^{\sigma\theta}_{\gamma}$ is reductive.
\end{theorem}

\begin{proof}
First, we need to show that the morphism from $G^{\sigma\theta}_{\gamma}$ to $C^{\tau}_{G^{\sigma}, \gamma\gamma^{-\theta}}$ given on $R$ points by
{\allowdisplaybreaks\begin{align*}
G^{\sigma\theta}_{\gamma}(R)&\rightarrow C^{\tau}_{G^{\sigma}, \gamma\gamma^{-\theta}}(R)\\
(x, y)&\mapsto x
\end{align*}}is an isomorphism. It is clearly injective. For surjectivity, if $x^{-1}\gamma\gamma^{-\theta} x^{\tau}=\gamma\gamma^{-\theta}$ then $$\gamma^{-1}x^{-1}\gamma=\gamma^{-\theta}x^{-\tau}\gamma^{\theta}=(\gamma^{-1}x^{-\sigma}\gamma)^{\theta}=(\gamma^{-1}x^{-1}\gamma)^{\theta}$$so $(x, (\gamma^{-1}x^{-1}\gamma)^{-1})\in G^{\sigma\theta}_{\gamma}(R).$

Notice that $C^{\tau}_{G^{\sigma},\gamma\gamma^{-\theta}}$ is the stabilizer of $\gamma\gamma^{-\theta} \in S(k)$ under the action of $G^{\sigma}$ on $S$ by $\tau$-conjugation.  Since the $G^{\sigma}$-orbit of $\gamma\gamma^{-\theta}$ is closed, it follows that $C^{\tau}_{G^{\sigma},\gamma\gamma^{-\theta}}$ is reductive \cite[2.1.10 (b)]{BR}.
\end{proof}

\section{Twisted relative orbital integrals}

Let $v$ be a place of $F$ and let
\begin{align*}
\mu_v:=\mu_{1v} \times \mu_{2v}  :  G^{\sigma}(F_v)\times G^{\theta}(F_v) \to \mathbb{C}^{\times}
\end{align*}be a unitary character. Let $\gamma_v\in G(F_v)$ be a relatively $\tau$-semisimple element. Since $(G_{\gamma_v}^{\sigma\theta})^{\circ}$ is reductive, the closed subgroup
$$(G^{\sigma\theta}_{\gamma_v})^{\circ}(F_v)\leq G^{\sigma}(F_v)\times G^{\theta}(F_v) $$is unimodular. Thus we can and do choose Haar measures $dx_v$, $dy_v$, and $dt_{\gamma_v}$ on $G^{\sigma}(F_v)$, $G^{\theta}(F_v)$ and $(G^{\sigma\theta}_{\gamma_v})^{\circ}(F_v)$, respectively, and form the right invariant Radon measure $\frac{dx_vdy_v}{dt_{\gamma_v}}$.

\begin{definition}
Let $\gamma_v\in G(F_v)$ be a relatively $\tau$-semisimple element. Then $\gamma_v$ is said to be a {\it relevant element } if $\mu_v$ is trivial on  $(G^{\sigma\theta}_{\gamma_v})^{\circ}(F_v)$.
\end{definition}

For $f_v\in C_c^{\infty}(G(F_v))$ and for a relevant element $\gamma_v$, the twisted relative orbital integral is defined by
\begin{equation}\label{TRO_v}
\mathrm{TRO}^{\mu_v}_{\gamma_v}(f_v):=\iint_{(G^{\sigma\theta}_{\gamma_v})^{\circ}(F_v)\backslash G^{\sigma}(F_v)\times G^{\theta}(F_v)}\mu_v(x_v, y_v^{-1})f_v(x_v^{-1}\gamma_v y_v)\frac{dx_vdy_v}{dt_{\gamma_v}}\,.
\end{equation}
The orbital integral \eqref{TRO_v} is absolutely convergent:

\begin{proposition}\label{TRO-conv}
The twisted relative orbital integral defined in \eqref{TRO_v} is absolutely convergent.
\end{proposition}
\begin{proof}
Recall that $|\mu_v|=1$ so it plays no role in absolute convergence. There is a surjective map \cite[\S 3]{RR}
\begin{align*}
C_c^{\infty}(G(F_v))&\longrightarrow C_c^{\infty}(S(F_v))\\
f_v&\longmapsto \psi_{f_v},
\end{align*}where $$\psi_{f_v}(g_vg_v^{-\theta}):=\int_{G^{\theta}(F_v)}f_v(g_vy_v)dy_v, \quad g_v\in G(F_v).$$Note that it is not necessarily true that all of the elements of $S(F_v)$ are of the form $g_vg_v^{-\theta}$ for $ g_v\in G(F_v)$. We have that
\begin{align}
\iint_{(G^{\sigma\theta}_{\gamma_v})^{\circ}(F_v)\backslash G^{\sigma}(F_v)\times G^{\theta}(F_v)}&|f_v(x_v^{-1}\gamma_v y_v)|\frac{dx_vdy_v}{dt_{\gamma_v}}\label{f-psi}\\
=&\int_{(C^{\tau}_{G^{\sigma}, \ga_v\ga_v^{-\theta}})^{\circ}(F_v)\backslash G^{\sigma}(F_v)}|\psi_{f_v}(g_v(\ga_v\ga_v^{-\theta})g_v^{-\tau})|dg_v\nonumber\\
=&\int_{(C^{\tau}_{G^{\sigma}, \ga_v\ga_v^{-\theta}})^{\circ}(F_v)\backslash G^{\sigma}(F_v)}|\psi_{f_v}(g_v\ga_v(g_v\ga_v)^{-\theta})|dg_v.\nonumber
\end{align}
The $G^{\sigma}(F_v)$-orbit of $\ga_v\ga_v^{-\theta}$ in $S(F_v)$ is closed by Lemma \ref{closed}. It follows that the support of the function
\begin{align*}
(C^{\tau}_{G^{\sigma}, \ga_v\ga_v^{-\theta}})^{\circ}(F_v)\backslash G^{\sigma}(F_v)&\longrightarrow \mathbb{C}\\
g_v&\longmapsto \psi_{f_v}(g_v(\ga_v\ga_v^{-\theta})g_v^{-\tau})
\end{align*}
is compact, so the convergence of the right hand side of \eqref{f-psi} is immediate.
\end{proof}

Let
{\allowdisplaybreaks\begin{align*}
\mu_1: & \,\,G^{\sigma}(\mathbb{A}_F)\longrightarrow \mathbb{C}^{\times},\\
\mu_2: & \,\,G^{\theta}(\mathbb{A}_F) \longrightarrow \mathbb{C}^{\times}
\end{align*}be unitary characters trivial on $G^{\sigma}(F)$ and $G^{\theta}(F)$, respectively, and set
$$\mu:= \mu_1\times\mu_2:G^{\sigma}(\mathbb{A}_F)\times G^{\theta}(\mathbb{A}_F)\longrightarrow \mathbb{C}^{\times}.$$
\begin{definition}
A relatively  $\tau$-semisimple element $\gamma \in G(\mathbb{A}_F)$ is said to be a {\it relevant element } if  $\mu$ is trivial on  $(G^{\sigma\theta}_{\gamma})^{\circ}(\mathbb{A}_F)$.
\end{definition}
For $f\in C_c^{\infty}(\mathbb{A}_F)$ and for a relevant element $\ga$, we define
\begin{equation}\label{TRO}
\mathrm{TRO}^{\mu}_{\gamma}(f):=\iint_{(G^{\sigma\theta}_{\gamma})^{\circ}(\mathbb{A}_F)\backslash G^{\sigma}(\mathbb{A}_F)\times G^{\theta}(\mathbb{A}_F)}\mu(x,y^{-1})f(x^{-1}\gamma y)\frac{dxdy}{dt_{\gamma}}\,,
\end{equation}where $dx$, $dy$, and $dt_{\gamma}$ are Haar measures on $G^{\sigma}(\mathbb{A}_F)$, $G^{\theta}(\mathbb{A}_F)$ and $(G^{\sigma\theta}_{\gamma})^{\circ}(\mathbb{A}_F)$, respectively.  When $G^{\sigma}$ is connected and $\gamma$ is strongly relatively $\tau$-regular,   Proposition \ref{TRO-conv} and Proposition \ref{one} below imply that \eqref{TRO} is absolutely convergent.

\begin{proposition}\label{one}
Suppose that $G^{\sigma}$ is connected and $\gamma \in G(F)$ is a relevant strongly relatively $\tau$-regular element. Then, for almost all finite places $v$, we have that $G^{\sigma}$, $G^{\theta}$, and $G$ are unramified at $v$ and we have that
\begin{equation}\label{one_v}
\mathrm{TRO}^{\mu_v}_{\gamma_v}(\mathrm{ch}_{K_v})=\iint_{G_{\gamma_v}^{\sigma\theta}(F_v)\cap (K^{\sigma}_v\times K^{\theta}_v)\backslash K^{\sigma}_v\times K^{\theta}_v}\frac{dx_vdy_v}{dt_{\gamma_v}}=1,
\end{equation}where $K_v\leq G(F_v)$ is a hyperspecial subgroup such that
$$
K^{\sigma}_v:=K_v\cap G^{\sigma}(F_v),\quad
K^{\theta}_v:=K_v\cap G^{\theta}(F_v),
$$and
$$G_{\gamma_v}^{\sigma\theta}(F_v)\cap (K^{\sigma}_v\times K^{\theta}_v)\leq G^{\sigma\theta}_{\gamma_v}(F_v)$$are all hyperspecial. Here $dx_v$, $dy_v$, and $dt_{\gamma_v}$ are normalized to give $K^{\sigma}_v$, $K^{\theta}_v$ and $G_{\gamma_v}^{\sigma\theta}(F_v)\cap (K^{\sigma}_v\times K^{\theta}_v)$ volume $1$, respectively. As usual, $\mathrm{ch}_{K_v}$ is the characteristic function of $K_v$.
\end{proposition}
\begin{proof}
Let $\gamma\in G(F)$ be strongly relatively $\tau$-regular and let $O(\gamma)$ denote its geometric relative $\tau$-class. We view $O(\gamma)$ as an affine $F$-subscheme of $G$. Note that the categorical quotient $G_{\gamma}^{\sigma\theta}\backslash (G^{\sigma}\times G^{\theta})$ exists as an affine scheme of finite type over $F$ (see \cite[Theorem 1.1 \S 1.2]{GIT}). There is a natural isomorphism
$$G_{\gamma}^{\sigma\theta}\backslash (G^{\sigma}\times G^{\theta})\cong O(\gamma)$$induced by the morphism $G^{\sigma}\times G^{\theta}\longrightarrow G$ given on $F$-algebras $R$ by
{\allowdisplaybreaks\begin{align}
G^{\sigma}(R)\times G^{\theta}(R)&\longrightarrow G(R)\label{useful-one}\\
(x, y)&\longmapsto x\gamma y^{-1}.\nonumber
\end{align}}Choose an embedding $G\longrightarrow \text{GL}_n(F)$ for some integer $n$ and let $\mathcal{G}$ be the scheme theoretic closure of $G$ in $\text{GL}_n(\mathcal{O}_F)$. Moreover, let
$$
\mathcal{O}(\gamma),\, \mathcal{G}^{\sigma},\, \mathcal{G}^{\theta}\subset \mathcal{G}\quad\text{and}\quad\mathcal{G}^{\sigma\theta}_{\gamma}\subset\mathcal{G}^{\sigma}\times\mathcal{G}^{\theta}
$$
be the scheme theoretic closures of $O(\gamma),\, G^{\sigma},\, G^{\theta}$ and $G^{\sigma\theta}_{\gamma}$, respectively. Then, for almost all finite places $v$ of $F$, the group schemes $\mathcal{G}^{\sigma},\, \mathcal{G}^{\theta},\,\mathcal{G}$ and $\mathcal{G}^{\sigma\theta}_{\gamma}$ are reductive and their $\mathcal{O}_{F_v}$-points are hyperspecial subgroups of their $F_v$-points \cite[\S 3.9]{T}. Here we are using the fact that $\ga$ is strongly relatively $\tau$-regular to conclude that $G^{\sigma\theta}_{\gamma}$ is connected.

Moreover, there is an integer $N$ such that the natural morphism
$$(\mathcal{G}^{\sigma}\times \mathcal{G}^{\theta})_{\mathcal{O}_F [N^{-1}]}\longrightarrow \mathcal{O}(\gamma)_{\mathcal{O}_F[N^{-1}]}$$is a smooth morphism between smooth schemes whose generic fiber is induced by \eqref{useful-one}.

Now let $v$ be a finite place of $F$ whose residual characteristic is prime to $N$. By enlarging $N$ if necessary, we can and do assume that $\gamma\in \mathcal{G}(\mathcal{O}_{F_v})$. We claim that the map
$$
\mathcal{G}^{\sigma}(\mathcal{O}_{F_v})\times \mathcal{G}^{\theta}(\mathcal{O}_{F_v})\longrightarrow \mathcal{O}(\gamma)(\mathcal{O}_{F_v})
$$
is surjective. Indeed, let $\gamma'\in \mathcal{O}(\gamma)(\mathcal{O}_{F_v}),$ let $\mathbb{F}_v$ be the residue field of $\mathcal{O}_{F_v},$ and let $\overline{\gamma}'\in \mathcal{O}(\gamma)(\mathbb{F}_v)$ be the image of $\gamma'$ under the reduction map. Consider the set of $(\overline{x}, \overline{y})\in \mathcal{G}^{\sigma}(\mathbb{F}_v)\times\mathcal{G}^{\theta}(\mathbb{F}_v)$ such that $\overline{x}\overline{\gamma}'\overline{y}^{-1}=\overline{\gamma}$. This is a (nonempty) $\mathcal{G}^{\sigma\theta}_{\gamma \mathbb{F}_v}$-homogeneous space. Since $\ga$ is strongly relatively $\tau$-regular, $\mathcal{G}^{\sigma\theta}_{\gamma \mathbb{F}_v}$ is connected and hence has an $\mathbb{F}_v$-point \cite[Theorem 1.9]{S1}. In other words, there is a $(x, y)\in \mathcal{G}^{\sigma}(\mathbb{F}_v)\times\mathcal{G}^{\theta}(\mathbb{F}_v)$ such that $x\overline{\gamma}'y^{-1}=\overline{\gamma}$. Since $\mathcal{G}^{\sigma}\times \mathcal{G}^{\theta}$ is smooth, this implies that there is a lift of $(x, y)$ to an element of $\mathcal{G}^{\sigma}(\mathcal{O}_{F_v})\times \mathcal{G}^{\theta}(\mathcal{O}_{F_v})$ \cite[Proposition 5]{BLR}.

Summing up, if the residual characteristic of a finite place $v$ of $F$ is large enough and $\gamma$ is an element of the hyperspecial subgroup $K:=\mathcal{G}(\mathcal{O}_{F_v})$, then any $\gamma'\in K$ in the geometric relative $\tau$-class  of $\gamma$ in fact satisfies $\gamma'\in K^{\sigma}\gamma K^{\theta}$, where $K^{\sigma}:=\mathcal{G}^{\sigma}(\mathcal{O}_{F_v})$ and $K^{\theta}:=\mathcal{G}^{\theta}(\mathcal{O}_{F_v})$.

Thus, for almost all $v$ we have that $\mathrm{TRO}_{\gamma_v}^{\mu_v}(\mathrm{ch}_{K_v})$ is equal to
\begin{equation}\label{useful-two}
\iint_{G_{\gamma_v}^{\sigma\theta}(F_v)\cap (K^{\sigma}_v\times K^{\theta}_v)\backslash (K^{\sigma}_v\times K^{\theta}_v)}\mu_v(x_v, y^{-1}_v)\frac{dx_vdy_v}{dt_{\gamma_v}}.
\end{equation}The character $\mu_v$ is trivial on $K_v$ for almost all $v$, which completes the proof.
\end{proof}

\section{Geometric expansion for the elliptic kernel}

In this section, we provide a geometric expansion of the relatively $\tau$-elliptic portions of the twisted relative trace formula. Throughout this section and for the rest of the paper, we assume that $G^{\sigma}$ is connected.

Recall the central subgroup $A_G$ and the Harish-Chandra subgroup ${}^1G(\mathbb{A}_F)$ defined in \eqref{ker}. To ease notation, set
\begin{equation}\label{2Gsigma-theta}
{}^2G^{\sigma}(\mathbb{A}_F):=G^{\sigma}(\mathbb{A}_F)\cap{}^1G(\mathbb{A}_F),\quad
{}^2G^{\theta}(\mathbb{A}_F):=G^{\theta}(\mathbb{A}_F)\cap{}^1G(\mathbb{A}_F)
\end{equation}and
\begin{equation}
A^{\sigma}_G:=A_G\cap G^{\sigma}(F\otimes_{\mathbb{Q}}\mathbb{R}),\quad
A^{\theta}_G:=A_G\cap G^{\theta}(F\otimes_{\mathbb{Q}}\mathbb{R}).
\end{equation}
Define
$$A:=\{ (z, z)\in A^{\sigma}_{G}\times A^{\theta}_{G} : z\in A^{\sigma}_{G}\cap A^{\theta}_{G} \}.$$

For every $\gamma\in G(F)$, define
\begin{equation}\label{2G}
{}^2G^{\sigma\theta}_{\gamma}(\mathbb{A}_F):=(G^{\sigma\theta}_{\gamma})^{\circ}(\mathbb{A}_F)\cap({}^2G^{\sigma}(\mathbb{A}_F)\times {}^2G^{\theta}(\mathbb{A}_F)).
\end{equation}Then we decompose
\begin{equation}\label{decom}
(G^{\sigma\theta}_{\gamma})^{\circ}(\mathbb{A}_F)\backslash (G^{\sigma}\times G^{\theta})(\mathbb{A}_F)\cong (A\backslash A^{\sigma}_{G}\times A^{\theta}_{G})\times ({}^2G^{\sigma\theta}_{\gamma}(\mathbb{A}_F)\backslash {}^2G^{\sigma}(\mathbb{A}_F)\times {}^2G^{\theta}(\mathbb{A}_F)).
\end{equation}
We now specify a choice of Haar measure on all of the groups appearing in \eqref{decom}, at least if $\ga$ is relatively $\tau$-elliptic. Whenever a reductive $F$-group $H$ appears, we give $H(\mathbb{A}_F)$, $A_H$, ${}^1H(\mathbb{A}_F)$ the measures that are used in the definition of the  Tamagawa number $\tau(H)=[H :H^{\circ}]\tau(H^{\circ})$. We then fix, once and for all, measures $dz_{\sigma}$ for $A^{\sigma}_{G}$, $dz_{\theta}$ for $A^{\theta}_{G}$, and stipulate that the isomorphisms
\begin{align*}
{}^2G^{\sigma}(\mathbb{A}_F)&\cong A_{G^{\sigma}}/A^{\sigma}_G\times{}^1G^{\sigma}(\mathbb{A}_F)\\
{}^2G^{\theta}(\mathbb{A}_F)&\cong A_{G^{\theta}}/A^{\theta}_G\times{}^1G^{\theta}(\mathbb{A}_F)
\end{align*}are measure preserving. Notice that if $\ga$ is relatively $\tau$-elliptic then
$${}^2G^{\sigma\theta}_{\gamma}(\mathbb{A}_F)={}^1(G^{\sigma\theta}_{\gamma})^{\circ}(\mathbb{A}_F).$$By our earlier convention, we have already endowed ${}^2G^{\sigma\theta}_{\gamma}(\mathbb{A}_F)$ and $A=A_{G^{\sigma\theta}_{\gamma}}$ with Haar measures. Let $dt$ be this Haar measure for $A$.  Altogether, this endows all of the groups occurring in \eqref{decom} with measures. Thus ${}^2G^{\sigma\theta}_{\gamma}(\mathbb{A}_F)$ is given the unique Haar measure such that
$$\text{vol}((G^{\sigma\theta}_{\gamma})^{\circ}(F)\backslash {}^2G^{\sigma\theta}_{\gamma}(\mathbb{A}_F))=\tau((G^{\sigma\theta}_{\gamma})^{\circ}).$$For the rest of this paper, when we form twisted relative orbital integrals we use these choices of measures.

For $f$ in $C_c^{\infty}(G(\mathbb{A}_F))$, define
\begin{equation}\label{f1}
f^1(x)=\int_{A\backslash (A^{\sigma}_{G}\times A^{\theta}_{G})}f(z_{\sigma}^{-1}z_{\theta}x)\frac{dz_{\sigma}dz_{\theta}}{dt}\,.
\end{equation}
Note that $f^1\in C_c^{\infty}({}^1G(\A_F))$. Then the elliptic kernel function is given by
\begin{equation}\label{kernel-geometric}
K_{f^1}^e(x,y)=\sum_{\gamma\in G(F)}f^1(x^{-1}\gamma y),
\end{equation}where the sum is over relatively $\tau$-regular elliptic $\gamma\in G(F)$. We can now state the geometric expansion of the relatively $\tau$-elliptic portion of the twisted relative trace formula.

\begin{theorem}\label{geometric}
Suppose that $G^{\sigma}$ is connected. Let $f=\otimes f_v\in  C_c^{\infty}(G(\mathbb{A}_F))$ and define $f^1$ as in \eqref{f1}. Suppose that there exist places $v_1, v_2$ of $F$ (not necessarily distinct) such  that $f_{v_1}$ is supported on relatively $\tau$-elliptic elements and $f_{v_2}$ is supported on strongly relatively $\tau$-regular elements. Then we have that
\begin{align*}
\sum_{\{\gamma\}}c^{\tau}(\gamma) \mathrm{TRO}^{\mu}_{\gamma}(f)
=\iint_{G^{\sigma}(F)\times G^{\theta}(F)\backslash {}^2G^{\sigma}(\mathbb{A}_F)\times {}^2G^{\theta}(\mathbb{A}_F)}\mu(x,y^{-1})K_{f^1}^e(x,y)dxdy,
\end{align*}where the sum is over the relevant strongly relatively $\tau$-regular elliptic   classes in $G(F)$. Here 
\begin{equation}\label{const}
c^{\tau}(\gamma):=\mathrm{vol}(G^{\sigma\theta}_{\gamma}(F)\backslash {}^2G^{\sigma\theta}_{\gamma}(\mathbb{A}_F)),
\end{equation}
and ${}^2G^{\sigma}(\mathbb{A}_F)$, ${}^2G^{\theta}(\mathbb{A}_F)$ and ${}^2G^{\sigma\theta}_{\gamma}(\mathbb{A}_F)$ are defined as in \eqref{2Gsigma-theta} and \eqref{2G}.
\end{theorem}

\begin{rem}
 In the theorem, we allow $v_1=v_2$. It follows from the proof of Theorem \ref{geometric} that the sum of the left is finite, and hence absolutely convergent, and the integral on
the right is absolutely integrable.
\end{rem}

\noindent In order to prove Theorem \ref{geometric}, we require the following result:

\begin{proposition}\label{finite-elliptic}
Suppose that $G^{\sigma}$ is connected.
Let $C\subset G(\mathbb{A}_F)$ be a compact subset such that $C_v$ is a hyperspecial subgroup of $G(F_v)$ for almost all finite places $v$ of $F$.
There are only finitely many $G^{\sigma}(F)\gamma G^{\theta}(F)\in \Gamma_r(F)$ with $\gamma$ strongly relatively $\tau$-regular such that
$$
G^{\sigma}(\mathbb{A}_F)\gamma G^{\theta}(\mathbb{A}_F)\cap C\neq \emptyset.
$$
\end{proposition}

\noindent Assuming Proposition \ref{finite-elliptic} for the moment, we now prove Theorem \ref{geometric}.

\begin{proof}[Proof of Theorem \ref{geometric}]
Suppose that $\ga$ is a relevant strongly relatively $\tau$-regular element. By Proposition \ref{one}, we have that
$$\mathrm{TRO}^{\mu}_{\gamma}(f)=\prod_{v}\mathrm{TRO}^{\mu_v}_{\gamma_v}(f_v) <\infty.$$Note that we have used the assumption that $G^{\sigma}$ is connected in order to apply Proposition \ref{one}.
Clearly $c^{\tau}(\gamma)<\infty$ by the relative $\tau$-ellipticity of $\gamma$. Letting $C$ be the closure of $\text{supp}(f)$ in Proposition \ref{finite-elliptic}, we see that
$$
\sum_{G^{\sigma}(F)\gamma G^{\theta}(F)\in \Gamma_r(F)} |c^{\tau}(\gamma)\mathrm{TRO}^{\mu}_{\gamma}(f)|<\infty.
$$Indeed, there are only finitely many nonzero summands. Therefore we derive that
{\allowdisplaybreaks\begin{align}
\sum_{\{\gamma\}}&c^{\tau}(\gamma) \mathrm{TRO}^{\mu}_{\gamma}(f)\label{TRO-eqn}\\
=&\sum_{\{\gamma\}} c^{\tau}(\gamma)\iint_{G_{\gamma}^{\sigma\theta}(\mathbb{A}_F)\backslash G^{\sigma}(\mathbb{A}_F)\times G^{\theta}(\mathbb{A}_F)}\mu(x, y^{-1})f(x^{-1}\gamma y)\frac{dxdy}{dt_{\gamma}}\nonumber\\
=&\sum_{\{\gamma\}} c^{\tau}(\gamma) \iint_{{}^2G_{\gamma}^{\sigma\theta}(\mathbb{A}_F)\backslash {}^2G^{\sigma}(\mathbb{A}_F)\times {}^2G^{\theta}(\mathbb{A}_F)}\mu(x, y^{-1})f^1(x^{-1}\gamma y)d\dot{x}d\dot{y}.\nonumber
\end{align}}Here $d\dot{x}d\dot{y}$ is the appropriate quotient measure, and the sums are over the relevant relatively $\tau$-regular elliptic semisimple classes.
 Notice that we have used our assumption that $\gamma$ is strongly relatively $\tau$-regular to  conclude that $(G_{\gamma}^{\sigma\theta})^{\circ}=G_{\gamma}^{\sigma\theta}$.  Also in the last equality of \eqref{TRO-eqn} we have used  \eqref{decom} and \eqref{f1}.

Notice that for irrelevant $\gamma$ the integral
$$
\iint_{G_{\gamma}^{\sigma\theta}(F)\backslash {}^2G^{\sigma}(\mathbb{A}_F)\times {}^2G^{\theta}(\mathbb{A}_F)}\mu(x, y^{-1})f^1(x^{-1}\gamma y)d\dot{x}d\dot{y}=0
$$
since it factors through
$$
\int_{G_{\gamma}^{\sigma\theta}(F)\backslash G_{\gamma}^{\sigma\theta}(\mathbb{A}_F)}\mu(g_1, g_2^{-1})dg_1dg_2
$$for appropriate measures $dg_1$ and $dg_2$. With this in mind, by the definition of the volume term $c^{\tau}(\gamma)$ in \eqref{const}, the last line in \eqref{TRO-eqn} is equal to
{\allowdisplaybreaks\begin{align*}
\sum_{\{\gamma\}} &\iint_{G_{\gamma}^{\sigma\theta}(F)\backslash {}^2G^{\sigma}(\mathbb{A}_F)\times {}^2G^{\theta}(\mathbb{A}_F)}\mu(x, y^{-1})f^1(x^{-1}\gamma y)d\dot{x}d\dot{y}\\
=&\iint_{G^{\sigma}(F)\times G^{\theta}(F)\backslash {}^2G^{\sigma}(\mathbb{A}_F)\times {}^2G^{\theta}(\mathbb{A}_F)}\mu(x, y^{-1})\sum_{\gamma\in G(F)}f^1(x^{-1}\gamma y)dxdy\\
=&\iint_{G^{\sigma}(F)\times G^{\theta}(F)\backslash {}^2G^{\sigma}(\mathbb{A}_F)\times {}^2G^{\theta}(\mathbb{A}_F)}\mu(x,y^{-1})K_{f^1}^e(x,y)dxdy.
\end{align*}}Here the sum in the second line is over strongly relatively $\tau$-regular elliptic elements by our assumptions on the support of $f$.
\end{proof}

\noindent We now need to prove Proposition \ref{finite-elliptic}.

\begin{proof}[Proof of Proposition \ref{finite-elliptic}]
Let $S$ be the subscheme of $G$ defined as in \eqref{S}.  We let $G^{\theta}$ act on $G$ by right multiplication, thus the map
$$
\begin{CD}
G @>{B_{\theta}}>>S
\end{CD}
$$is constant on $G^{\theta}$-orbits. Moreover, $B_{\theta}$ intertwines the action of $G^{\sigma}$ on $G$ by left multiplication with the action of $G^{\sigma}$ on $S$ by $\tau$-conjugation. Write
$$G^{\sigma}\backslash S=\mathrm{Spec} (\Gamma(S, \mathcal{O}_S)^{G^{\sigma}})\,.$$
Let $P: S \to G^{\sigma}\backslash S$ be the quotient map and consider the composite map
$$
\begin{CD}
G(\mathbb{A}_F) @>{B_{\theta}}>>S(\mathbb{A}_F) @>{P}>>(G^{\sigma}\backslash S)(\mathbb{A}_F)\,.
\end{CD}
$$Observe that
\begin{equation}\label{C}
(P\circ B_{\theta})(C)\cap (G^{\sigma}\backslash S)(F)
\end{equation}is finite since $(G^{\sigma}\backslash S)(F)$ is discrete in $(G^{\sigma}\backslash S)(\mathbb{A}_F)$ if we give the later set the canonical adelic topology. By Lemma \ref{closed}, if $\ga$ is a relatively $\tau$-semisimple element, then the $G^{\sigma}$ orbit of $\ga\ga^{-\theta}$ is a zariski closed subset of $S$. Thus if $\gamma, \gamma'\in G(F)$ are relatively $\tau$-semisimple and have the same  image in $(G^{\sigma}\backslash S)(\mathbb{A}_F)$, then
$$g^{-1}\gamma\gamma^{-\theta}g^{\tau}=\gamma'\gamma'^{-\theta}$$for some $g\in G^{\sigma}(\overline{F})$ (see \cite[\S 0.2]{GIT}). Since \eqref{C} is finite,  it follows that there are only finitely many geometric relative $\tau$-classes in $\Gamma_r^{\mathrm{ge}}(F)$ consisting of relatively $\tau$-semisimple elements that intersect $C$.

Thus to complete the proof of Proposition \ref{finite-elliptic}, it suffices to show that there are only finitely many relative $\tau$-classes that intersect $C$ and are in a fixed geometric relative $\tau$-class consisting of strongly relatively $\tau$-regular elements.  Let $G^{\sigma}(F) \gamma G^{\theta}(F)$ be a relative $\tau$-class with $\gamma \in G(F)$ strongly relatively $\tau$-regular.
From the proof of Proposition \ref{one}, there is a finite set of places $V$ of $F$ such that if $v\notin V$ then $\gamma_v\in C_v$ and if $\gamma'_v\in C_v$ is in the geometric relative $\tau$-class of $\gamma_v$, then 
$$
G^{\sigma}(F_v)\gamma_vG^{\theta}(F_v)=G^{\sigma}(F_v)\gamma'_vG^{\theta}(F_v).
$$
On the other hand, for $v\in V$, a standard Galois cohomology argument \cite[\S 1.8]{L} implies that there are only finitely many elements of $\Gamma_r(F_v)$ in a given geometric relative $\tau$-class in $\Gamma_r^{\mathrm{ge}}(F_v)$. In summary, there are only finitely many 
$$
G^{\sigma}(\A_F)\gamma' G^{\theta}(\A_F) \in \Gamma_r(\mathbb{A}_F)
$$
 that intersect $C$ such that $\ga_v\in G^{\sigma}(\overline{F}_v)\gamma'_vG^{\theta}(\overline{F}_v)$ for all places $v$. By another standard Galois cohomology argument, the set of elements in $\Gamma_r(F)$ contained in a given element of $\Gamma_r(\mathbb{A}_F)$ injects into the cohomology group $\mathfrak{E}(G^{\sigma\theta}_{\gamma}, G^{\sigma}\times G^{\theta};\mathbb{A}_F/F)$ (compare \cite[Lemme 1.8.5]{L}). For the definition of  $\mathfrak{E}(G^{\sigma\theta}_{\gamma}, G^{\sigma}\times G^{\theta};\mathbb{A}_F/F)$),  see \cite[\S 1.8]{L}. The group $\mathfrak{E}(G^{\sigma\theta}_{\gamma}, G^{\sigma}\times G^{\theta};\mathbb{A}_F/F)$ is finite by \cite[Proposition 1.8.4]{L}. This completes the proof.
\end{proof}

\section{Spectral expansion}\label{Se}
Recall that
$$
\mu:=\mu_1\times \mu_2: G^{\sigma}(\mathbb{A}_F)\times G^{\theta}(\mathbb{A}_F)\longrightarrow \mathbb{C}^{\times},
$$
where $\mu_1$ and $\mu_2$ are unitary characters trivial on $G^{\sigma}(F)$ and $G^{\theta}(F)$, respectively. Let $L^2_0(G(F)\backslash {}^1G(\mathbb{A}_F))$ be the subspace of cusp forms in $L^2(G(F)\backslash {}^1G(\mathbb{A}_F))$.
Let $\pi$ be an automorphic representation of $G(\mathbb{A}_F)$ with unitary central character whose restriction to ${}^1G(\mathbb{A}_F)$ admits a realization in $L^2_0(G(F)\backslash {}^1G(\mathbb{A}_F))$.  For smooth $\phi$ in the $\pi$-isotypic subspace of $L^2_0(G(F)\backslash {}^1G(\mathbb{A}_F))$, the $(G^{\sigma}, \mu_1)$ and $(G^{\theta}, \mu_2)$-period integrals are defined by
\begin{align}
\mathcal{P}^{\sigma}_{\mu_1}(\phi)&:=\int_{ G^{\sigma}(F)\backslash {}^2G^{\sigma}(\mathbb{A}_F)}\mu_1(x)\phi(x)dx,\label{Period_int1}\\
\mathcal{P}^{\theta}_{\mu_2}(\phi)&:=\int_{ G^{\theta}(F)\backslash {}^2G^{\theta}(\mathbb{A}_F)}\mu_2(y)\phi(y)dy.\label{Period_int2}
\end{align}Here $${}^2G^{\sigma}(\mathbb{A}_F):=G^{\sigma}(\mathbb{A}_F)\cap{}^1G(\mathbb{A}_F)\quad\text{and} \quad{}^2G^{\theta}(\mathbb{A}_F):=G^{\theta}(\mathbb{A}_F)\cap{}^1G(\mathbb{A}_F)$$ and $dx$ (resp. $dy$) is the choice of Haar measure on
$G^{\sigma}(F)\backslash {}^2G^{\sigma}(\mathbb{A}_F)$ (resp. $G^{\theta}(F)\backslash {}^2G^{\theta}(\mathbb{A}_F)$) given in \eqref{2Gsigma-theta}. Both \eqref{Period_int1} and \eqref{Period_int2} converge absolutely  \cite[\S2 Proposition 1]{AGR}.

We now derive a spectral expansion of the integral appearing in Theorem \ref{geometric}. We follow Gelbart \cite[\S 1.2]{G} in our exposition. Let $R$ be the regular representation of ${}^1G(\mathbb{A}_F)$ on $L^2(G(F)\backslash {}^1G(\mathbb{A}_F))$. Since $G(F)\backslash {}^1G(\mathbb{A}_F)$ is not  in general compact,  $R$ does not decompose discretely. We write  $R_0$ for the subrepresentation of $R$ on the  invariant subspace $L^2_0(G(F)\backslash {}^1G(\mathbb{A}_F))\leq L^2(G(F)\backslash {}^1G(\mathbb{A}_F))$.  Then $R_0$ decomposes discretely.
  We make the following:
\begin{convention}
An  automorphic representation of ${}^1G(\mathbb{A}_F)$ is the restriction to ${}^1G(\mathbb{A}_F)$ of an automorphic representation of $G(\mathbb{A}_F)$. Two such representations are equivalent if they are equivalent as representations of ${}^1G(\mathbb{A}_F)$.
\end{convention}
\noindent Note that, with this convention, an automorphic representation  $\pi$ of $^1G(\mathbb{A}_F)$ admitting a realization in $R_0$ is unitary and hence has unitary central character.

For $f\in C_c^{\infty}(G(\mathbb{A}_F))$, let $f^1$ be defined as in \eqref{f1} and we write $R_0(f)=\sum_{\pi}m_{\pi}\pi(f^1)$ where the sum is over equivalence classes of irreducible subrepresentations of $R_0$ and $m_{\pi}$ is the mutiplicity of $\pi$ in $R_0$.
Equivalently $L_0^2(G(F)\backslash {}^1G(\mathbb{A}_F))=\sum V_{\pi}$ where $V_{\pi} \leq L_0^2(G(F) \backslash {}^1G(\mathbb{A}_F))$ is an invariant subspace realizing the $m_{\pi}$ copies of $\pi$. Let $\mathcal{B}_{\pi}$ be an orthonormal basis of $V_{\pi}$ with respect to the pairing
\begin{align}
V_{\pi}\times V_{\pi}&\longrightarrow \mathbb{C}\nonumber\\
(\phi_1, \phi_2) &\longmapsto\int_{G(F)\backslash {}^1G(\mathbb{A}_F)}\phi_1(x)\overline{\phi_2(x)}dx, \label{pair}
\end{align}with $dx$ the Tamagawa measure. Then each $m_{\pi}\pi(f^1)$ is an integral operator in $V_{\pi}$ with $L^2$-kernel
\begin{align} \label{smooth-kernel}
K_{\pi(f^1)}(x,y)=\sum_{\phi\in\mathcal{B}_{\pi}} (\pi(f^1)\phi)(x)\overline{\phi(y)}.
\end{align}
We note that, a priori, this sum only converges in $L^2_0(G(F) \backslash {}^1G(\mathbb{A}_F) \times G(F) \backslash {}^1G(\mathbb{A}_F))$.  However, Arthur has shown that there exists a (unique) square-integrable function that is smooth in $x$ and $y$ separately and represents this kernel \cite[Lemma 4.5 and Lemma 4.8]{A}.  We will denote this function by $K_{\pi(f^1)}(x,y)$.

We define 
\begin{equation}\label{TRT}
\mathrm{TRT}_{G^{\sigma},G^{\theta}}^{\mu}(\pi(f^1)):=\iint_{G^{\sigma}(F)\times G^{\theta}(F)  \backslash {}^2G^{\sigma}(\A_F) \times {}^2G^{\theta}(\A_F)}\mu(x,y^{-1})K_{\pi(f^1)}(x,y)dxdy.
\end{equation}
The integrand is smooth separately in $x$ and $y$ and is in the cuspidal subspace, hence rapidly decreasing by \cite[\S 4]{HC}.  Therefore it is absolutely convergent by \cite[\S2 Proposition 1]{AGR}.  Let $K_{\infty}$ be a maximal compact subgroup of $G(\A_{F,\infty})$.  We note that if $f^1$ is $K_{\infty}$-finite, then the sum over $\phi$ in \eqref{smooth-kernel} can be taken to be  a finite sum over $K_{\infty}$-finite functions
\cite[Above Lemma 4.5]{A}.  Therefore, in this case
$$
\mathrm{TRT}_{G^{\sigma},G^{\theta}}^{\mu}(\pi(f^1))=\sum_{\phi \in \mathcal{B}_{\pi}} \mathcal{P}^{\sigma}_{\mu_1}(\pi(f^1)\phi) \overline{\mathcal{P}^{\theta}_{\mu_2}(\phi)}.
$$
In general, if $\mathrm{TRT}_{G^{\sigma},G^{\theta}}^{\mu}(\pi(f^1)) \neq 0$, then $\pi$ is both $(G^{\sigma},\mu_1)$ and $(G^{\theta},\mu_2)$-distinguished.

If $R(f^1)$ has  image in $L_0^2(G(F)\backslash {}^1G(\mathbb{A}_F))$, its kernel has $L^2$-expansion
\begin{equation}\label{kernel-spectral}
K_{f^1}(x,y)=\sum_{\pi}K_{\pi(f^1)}(x,y),
\end{equation}
where the sum is over equivalence classes of automorphic representations $\pi$ of $^{1}G(\mathbb{A}_F)$ admitting a realization in $R_0$. Again, this kernel is represented by a (unique) square integrable function that is smooth in $x$ and $y$ separately \cite[Lemma 4.5 and Lemma 4.8]{A}.  We denote this function by $K_{f^1}(x,y)$ as well.  Then, by \cite[\S 4]{A}, the equation in \eqref{kernel-spectral} holds pointwise.

\begin{proposition}\label{spectral}
Let $f^1 \in  C_c^{\infty}({}^1G(\mathbb{A}_F))$ be defined as \eqref{f1}. Assume that the integral operator $R(f^1)$ has image in $L_0^2(G(F)\backslash {}^1G(\mathbb{A}_F))$. Then we have that
\begin{equation}\label{spectral-eqn}
\iint_{G^{\sigma}(F)\times G^{\theta}(F) \backslash {}^2G^{\sigma}(\mathbb{A}_F)\times {}^2G^{\theta}(\mathbb{A}_F)}\mu(x,y^{-1})K_{f^1}(x,y)dxdy=\sum_{\pi}\mathrm{TRT}_{G^{\sigma},G^{\theta}}^{\mu}(\pi(f^1)),
\end{equation}where the sum on the right is over equivalence classes of irreducible automorphic representations $\pi$ of ${}^1G(\mathbb{A}_F)$ occurring in $R_0$.  The integrand is absolutely integrable and the sum is absolutely convergent.
\end{proposition}

\begin{proof}
As noted above, $K_{f^1}(x,y)$ is a smooth function that is (by definition) in the closed subspace $L^2_0(G(F) \backslash {}^1G(\A_F) \times G(F) \backslash {}^1G(\A_F))$.  It is therefore rapidly decreasing \cite[\S 4]{HC}.  Thus, by \cite[\S 2, Proposition 1]{AGR} the integral on the left in \eqref{spectral-eqn} is absolutely convergent. Formally, we have that
{\allowdisplaybreaks\begin{align*}
&\iint_{G^{\sigma}(F)\times G^{\theta}(F) \backslash {}^2G^{\sigma}(\mathbb{A}_F)\times {}^2G^{\theta}(\mathbb{A}_F)} \mu(x,y^{-1})K_{f^1}(x,y)dxdy\\
&=\iint_{G^{\sigma}(F)\times G^{\theta}(F) \backslash {}^2G^{\sigma}(\mathbb{A}_F)\times {}^2G^{\theta}(\mathbb{A}_F)}\mu(x,y^{-1})\sum_{\pi}K_{\pi(f^1)}(x,y)dxdy\\
&=\sum_{\pi}\iint_{G^{\sigma}(F)\times G^{\theta}(F) \backslash {}^2G^{\sigma}(\mathbb{A}_F)\times {}^2G^{\theta}(\mathbb{A}_F)}\mu(x,y^{-1})K_{\pi(f^1)}(x,y)dxdy\\
&=\sum_{\pi}\mathrm{TRT}_{G^{\sigma},G^{\theta}}^{\mu}(\pi(f^1)).
\end{align*}}In order to make this rigorous,  we must justify switching the sum and integral.  As noted above, $K_{f^1}(x,y)=\sum_{\pi}K_{\pi(f_1)}(x,y)$ pointwise.  By the Lebesgue dominated convergence theorem it suffices to prove that
$\sum_{\pi} |K_{\pi(f^1)}(x,y)|$
is integrable over $G^{\sigma}(F)\times G^{\theta}(F)\backslash {}^2G^{\sigma}(\mathbb{A}_F)\times {}^2G^{\theta}(\mathbb{A}_F)$.

Choose an integer $m >0$.  If $m$ is sufficiently large, then there exists $K_{\infty}$-finite functions $g_i \in C_c^m({}^1G(\A_F))$ and an element $Z$ in the universal enveloping algebra $\mathfrak{U}(\mathrm{Lie}({}^1G(\A_{F,\infty})) \otimes \CC)$ fixed by $K_{\infty}$ (?) with the following property: If $r_0 > \mathrm{deg}(Z)$ and $h \in C_c^{r_0}({}^1G(\A_F))$ we can write
$$
h=\sum_{i=1}^2 h_i* g_i,
$$
where $h_1=h * Z$ and $h_2=h$ \cite[Corollary 4.2]{A}.  Notice that the $h_i$ are $K_{\infty}$-finite if $h$ is $K_{\infty}$-finite.

Write
$$
g^*(x):=\overline{g(x^{-1})}
$$
for compactly supported continuous functions $g \in C_c({}^1G(\A_F))$.
Assume for the moment that $h \in C_c^{r}({}^1G(\A_F))$ is $K_{\infty}$-finite for some $r \geq r_0$.
Applying the Cauchy-Schwarz inequality several times we have
\begin{align*}
|K_{\pi(h)}(x,y)|&=\Big|\sum_{\phi \in \mathcal{B}_{\pi}}(\pi(h)\phi)(x) \overline{\phi(y)}\Big|=\Big|\sum_{i=1}^2 \sum_{\phi \in \mathcal{B}_{\pi}} (\pi(h_i)\phi)(x) \overline{(\pi(g_i^*)\phi)(y)}\Big|\\
&\leq \sum_{i=1}^2 \Big|\sum_{\phi \in \mathcal{B}_{\pi}} (\pi(h_i)\phi)(x)\Big|\Big|\sum_{\phi \in \mathcal{B}_{\pi}}\overline{(\pi(g_i^*)\phi)(y)}\Big|\\
&\leq \sum_{i=1}^2\Big(\sum_{\phi \in \mathcal{B}_{\pi}}|(\pi(h_i)\phi)(x)|\Big) \Big( \sum_{\phi \in \mathcal{B}_{\pi}}|(\pi(g_i^*)\phi)(y)| \Big)\\
&=\sum_{i=1}^2 K_{\pi(h_i*h_i^*)}(x,x) K_{\pi(g_i^* * g_i)}(y,y).
\end{align*}
Here the $\phi \in \mathcal{B}_{\pi}$ are assumed to be $K_{\infty}$-finite (and hence smooth) and the sums can be taken to be finite because we have assumed that $h$ is $K_{\infty}$-finite. We note that $K_{\pi(h_i*h_i^*)}(x,x)$ and $K_{\pi(g_i^* * g_i)}(y,y)$ are nonnegative functions.  Thus, applying the Cauchy-Schwarz inequality again we have
\begin{align} \label{last-ineq}
\sum_{\pi}|K_{\pi(h)}(x,y)| &\leq \sum_{\pi} \sum_{i=1}^2K_{\pi(h_i*h_i^*)}(x,x) K_{\pi(g_i^* * g_i)}(y,y) \\&\leq \nonumber \sum_{i=1}^2\Big(\sum_{\pi}K_{\pi(h_i*h_i^*)}(x,x)\Big) \Big(\sum_{\pi} K_{\pi(g_i^* * g_i)}(y,y)\Big).
\end{align}

We claim that the estimate \eqref{last-ineq} is valid without assuming that $h$ is $K_{\infty}$-finite.
For each $r$, give $C_c^r({}^1G(\A_F))$ the topology explained in \cite[\S 2]{A}.  Then there exists a continuous seminorm $\|\cdot\|_{0}$ on $C_c({}^1G(\A_F))$ and a positive integer $N$ depending only on $G$ such that for each $K_{\infty}$-finite $h \in C_c^{r}({}^1G(\A_F))$ with $r\geq r_0$ we have
$$
\Big(\sum_{\pi}K_{\pi(h_i*h_i^*)}(x,x)\Big) \Big(\sum_{\pi} K_{\pi(g_i^* * g_i)}(y,y)\Big) \leq \|h_i*h_i^*\|_0^{1/2}\|g_i^* *g_i\|_0^{1/2}\|x\|^N\|y\|^N
$$and hence
$$
\sum_{\pi}|K_{\pi(h)}(x,y)| \leq \sum_{i=1}^2\|h_i*h_i^*\|_0^{1/2}\|g_i^* *g_i\|_0^{1/2}\|x\|^N\|y\|^N
$$(see \cite[Proof of Lemma 4.4]{A}).
Here $\|\cdot\|$ is the height function on ${}^1G(\A_F)$ given in \cite[\S 2]{A}.  It follows that for fixed $(x,y)$ the real-valued functions
\begin{align} \label{funcs}
h \mapsto \sum_{\pi}|K_{\pi(h)}(x,y)| \textrm{ and }h \mapsto \sum_{i=1}^2\Big(\sum_{\pi}K_{\pi(h_i*h_i^*)}(x,x)\Big) \Big(\sum_{\pi} K_{\pi(g_i^* * g_i)}(y,y)\Big)
\end{align}
are continuous on the $K_{\infty}$-finite subset of $C_c^{r}({}^1G(\A_F))$.  If they admitted continuous extensions to $C_c^{r}({}^1G(\A_F))$, then our claim would be proven.  Due to the fact that $K_{\infty}$-finite functions in $C_c^r({}^1G(\A_F))$ are not in general dense in $C_c^{r}({}^1G(\A_F))$,
we are not sure whether this statement is true \cite[p.~931]{A}.  Instead, following Arthur, we view 
a general $h \in C^r_c({}^1G(\A_F))$ as an element of $C_c^{r-l_0}({}^1G(\A_F))$ and approximate it in $C_c^{r-l_0}({}^1G(\A_F))$ with a sequence of $K_{\infty}$-finite functions.  Here $l_0$ is a certain positive integer.  This will allow us to deduce the claim.

Let $x$ and $y$ be fixed and let $h \in C_c^{r}({}^1G(\A_F))$; we do not assume $h$ is $K_{\infty}$-finite.  We assume that $r>l_0$, where $l_0$ is the integer discussed after the proof of \cite[Lemma 4.4]{A}; it depends only on $G$.  For each such $h$ there is a countable set of $K_{\infty}$-finite functions $h_{\omega} \in C_c^{r-l_0}({}^1G(\A_F))$ such that
$\sum_{\omega}h_{\omega}$ converges absolutely to $h$ in $C_c^{r-l_0}({}^1G(\A_F))$ and if $\|\cdot\|_?$ is any continuous seminorm on $C_c^{r-l_0}({}^1G(\A_F))$ then
$$
f \mapsto \sum_{\omega} \|f_{\omega}\|_?
$$
is again a continuous seminorm on $C_c^{r}({}^1G(\A_F))$ (see after the proof of \cite[Lemma 4.4]{A}).
Applying this with
\begin{align*}
\|h\|_{r_0}=\sum_{i=1}^2\|h_i *h_i^*\|_0^{1/2}\|g_i^**g_i\|_0^{1/2} \|x\|^N\|y\|^N 
\end{align*}
 implies that the functions in \eqref{funcs} extend to define continuous functions on $C_c^{r-l_0}({}^1G(\A_F))$.  This implies our claim that \eqref{last-ineq} is valid for all $h \in C_c^{r-l_0}({}^1G(\A_F))$.

Choose a positive integer $L>0$.  Enlarging $m$ if necessary in order to make $g_1 \in C^m_c({}^1G(\A_F))$ smoother,  
we apply the proof of \cite[Theorem 14.1(a)]{Asurvey} (see also \cite[\S 2]{A2}) to conclude that 
$$
\sum_{\pi} |\Lambda_2^TK_{\pi(h_i*h_i^*)}(x,x)| \leq C(h_i) \|x\|^{-L}\quad\text{and}
$$
$$
\sum_{\pi}|\Lambda_2^TK_{\pi(g_i^*g_i)}(y,y)| \leq C(g_i) \|y\|^{-L},
$$
where $\Lambda_2^T$ is Arthur's truncation operator (see \cite[\S 1]{A2}), $C(h_i), C(g_i)$ are constants, and $T$ is assumed to be sufficiently regular.  On the other hand, 
$$
\Lambda_2^TK_{\pi(h_i*h_i^*)}(x,x)=K_{\pi(h_i*h_i^*)}(x,x),\quad\Lambda_2^TK_{\pi(g_i^**g_i)}(y,y)=K_{\pi(g_i^**g_i)}(y,y)
$$
for unitary $\pi$ realized in $L_0^2(G(F)\backslash {}^1G(\mathbb{A}_F))$ \cite[\S 1]{A2}.  Thus
\begin{align*}
\sum_{\pi}K_{\pi(h_i*h_i^*)}(x,x) \quad\textrm{and}\quad \sum_{\pi}K_{\pi(g_i* g_i^*)}(y,y)
\end{align*}
are both rapidly decreasing.  
Applying the proof of \cite[\S 2 Proposition 1]{AGR} we see that
$$
\iint_{G^{\sigma}(F)\times G^{\theta}(F)\backslash  {}^2G^{\sigma}(\A_F)  \times{}^2G^{\theta}(\A_F)}\Big(\sum_{\pi}K_{\pi(h_i*h_i^*)}(x,x)\Big)\Big( \sum_{\pi}K_{\pi(g_i^**g_i)}(y,y)\Big)dxdy<\infty.
$$
In view of \eqref{last-ineq} an application of the  Lesbesgue dominated convergence theorem completes the proof of the proposition.
\end{proof}

\begin{lemma}\label{L_0}
Suppose that $f=\otimes'_vf_v\in C_c^{\infty}(G(\mathbb{A}_F))$. If $f_v$ is $F$-supercuspidal for some finite place $v$, then $R(f^1)$ has image in $L_0^2(G(F)\backslash {}^1G(\mathbb{A}_F))$.
\end{lemma}

\begin{proof}
For any $\psi\in L^2(G(F)\backslash {}^1G(\mathbb{A}_F))$, and for the unipotent radical $N$ of any proper parabolic subgroup of $G$, we have that
{\allowdisplaybreaks\begin{align*}
&\int_{N(F)\backslash N(\mathbb{A}_F)}R(f^1)\psi(n^{-1}x)dn\\
&=\int_{N(F)\backslash N(\mathbb{A}_F)}\Big(\int_{N(F)\backslash{}^1G(\mathbb{A}_F)}\sum_{y\in N(F)}f^1(x^{-1}nyg)\psi(g)dg\Big)dn\\
&=\int_{N(F)\backslash {}^1G(\mathbb{A}_F)}\Big(\int_{N(\mathbb{A}_F)}f^1(x^{-1}ng)dn\Big)\psi(g)dg=0,
\end{align*}}since we have that
$$\int_{N(\mathbb{A}_F)}f^1(x^{-1}ng)dn=\prod_{v| \infty}\int_{N(F_v)}f^1_v(x_v^{-1}n_vg_v)dn_v\prod_{v\nmid \infty}\int_{N(F_v)}f_v(x_v^{-1}n_vg_v)dn_v=0$$for all $x\in {}^1G(\mathbb{A}_F)$.
\end{proof}

\noindent We now restate and prove the main theorem of this paper.
\begin{theorem}
Suppose that $G^{\sigma}$ is connected. Let $f=\otimes_v f_v\in C_c^{\infty}(G(\mathbb{A}_F))$ be such that there exist places $v_1, v_2, v_3$ of $F$ (not necessarily distinct) such that
\begin{enumerate}
\item{$f_{v_1}$ is supported on relatively $\tau$-elliptic elements of $G(F_{v_1})$,}
\item{$f_{v_2}$ is supported on strongly relatively $\tau$-regular elements of $G(F_{v_2})$,}
\item{$f_{v_3}$ is $F$-supercuspidal.}
\end{enumerate}
Then we have that
\begin{align}\label{main-eqn}
\sum_{\{\gamma\}}c^{\tau}(\gamma) \mathrm{TRO}^{\mu}_{\gamma}(f)=\sum_{\pi}\mathrm{TRT}_{G^{\sigma}, G^{\theta}}^{\mu}(\pi(f^1)),
\end{align}where the sum on the left is over relevant strongly relatively $\tau$-regular elliptic  classes in $G(F)$ and the sum on the right is over equivalence classes of automorphic representations $\pi$ of ${}^1G(\mathbb{A}_F)$ admitting a realization in $L^2_0(G(F)\backslash {}^1G(\mathbb{A}_F))$. 
\end{theorem}

\begin{proof}
Note that $K_{f^1}^e(x,y)=K_{f^1}(x,y)$ by our assumptions on support of $f$. Hence the theorem follows immediately from Theorem \ref{geometric} and Proposition \ref{spectral}.
\end{proof}

\section*{Acknowledgments}
I would like to thank the referees for their useful comments and corrections. I also thank Jayce Getz for his encouragement and helpful comments.


\end{document}